\documentclass[12pt]{amsart}

\title{Periodic points of rational functions over finite fields}
\author{Derek Garton}
\address{Fariborz Maseeh Department of Mathematics and Statistics, Portland State University}
\email{\href{mailto:gartondw@pdx.edu}{gartondw@pdx.edu}}
\date{\today}

\makeatletter
\@namedef{subjclassname@2020}{\textup{2020} Mathematics Subject Classification}
\makeatother

\subjclass[2020]{Primary 37P05;
Secondary 37P25, 37P35, 11T06, 13B05}

\keywords{Arithmetic Dynamics, Periodic Points, Finite Fields, Galois Theory}

\usepackage{amssymb,url,MnSymbol,colonequals,tikz-cd}

\usepackage{hyperref}
\hypersetup{breaklinks=true,
colorlinks=true,
urlcolor=blue,
linkcolor=cyan
}

\usepackage[inline,shortlabels]{enumitem}

\usepackage[top=1in,bottom=1in,left=1in,right=1in]{geometry}

\usepackage[capitalise,nameinlink]{cleveref}

\newcommand{\Z}{\ensuremath{\mathbb{Z}}}

\renewcommand{\P}{\ensuremath{\mathbb{P}}}
\newcommand{\Q}{\ensuremath{\mathbb{Q}}}
\newcommand{\R}{\ensuremath{\mathbb{R}}}
\newcommand{\F}{\ensuremath{\mathbb{F}}}

\newcommand{\lv}{\ensuremath{\left\vert}}
\newcommand{\rv}{\ensuremath{\right\vert}}

\newcommand{\lp}{\ensuremath{\left(}}
\newcommand{\rp}{\ensuremath{\right)}}
\newcommand{\lb}{\ensuremath{\left\{}}
\newcommand{\rb}{\ensuremath{\right\}}}
\newcommand{\lc}{\ensuremath{\left[}}
\newcommand{\rc}{\ensuremath{\right]}}

\newcommand{\p}{\ensuremath{\mathfrak{p}}}
\newcommand{\m}{\ensuremath{\mathfrak{m}}}

\newcommand{\q}{\ensuremath{\mathfrak{q}}}
\newcommand{\qover}{\ensuremath{\mathfrak{Q}}}

\newcommand{\av}{\ensuremath{\mathbf{a}}}

\DeclareMathOperator{\Gal}{Gal}
\DeclareMathOperator{\Aut}{Aut}
\DeclareMathOperator{\Per}{Per}
\DeclareMathOperator{\Char}{char}

\DeclareMathOperator{\id}{id}
\DeclareMathOperator{\Frac}{Frac}
\DeclareMathOperator{\Spec}{Spec}

\DeclareMathOperator{\Fix}{F}
\DeclareMathOperator{\fix}{f}

\DeclareMathOperator{\content}{content}

\theoremstyle{plain}
\newtheorem{theorem}{Theorem}[section]
\newtheorem{lemma}[theorem]{Lemma}

\newtheorem{fact}[theorem]{Fact}

\newtheorem*{namedthm}{\namedthmname}
\newcounter{namedthm}

\theoremstyle{remark}

\newtheorem{example}[theorem]{Example}

\theoremstyle{definition}

\makeatletter
\newenvironment{named}[1]
  {\def\namedthmname{#1}%
   \refstepcounter{namedthm}%
   \namedthm\def\@currentlabel{#1}}
  {\endnamedthm}
\makeatother

\begin{document}

\begin{abstract}
For $q$ a prime power and $\phi$ a rational function with coefficients in $\F_q$, let $p(q,\phi)$ be the proportion of $\P^1\lp\F_q\rp$ that is periodic with respect to $\phi$.
And if $d$ is a positive integer, let $Q_d$ be the set of prime powers coprime to $d!$ and let $\mathcal{P}(d,q)$ be the expected value of $p(q,\phi)$ as $\phi$ ranges over rational functions with coefficients in $\F_q$ of degree $d$.
We prove that if $d$ is a positive integer no less than $2$, then $\mathcal{P}(d,q)$ tends to 0 as $q$ increases in $Q_d$.
This theorem generalizes our previous work, which held only for quadratic polynomials, and only in fixed characteristic.
To deduce this result, we prove a uniformity theorem on specializations of dynamical systems of rational functions with coefficients in certain finitely-generated algebras over residually finite Dedekind domains.
This specialization theorem generalizes our previous work, which held only for algebras of dimension one.
\end{abstract}

\maketitle
\tableofcontents

\section{Introduction}
\label{intro}

A \emph{\textup{(}discrete\textup{)} dynamical system} is a pair $\lp S,\phi\rp$ consisting of a set $S$ and a function $\phi\colon S\to S$.
For notational convenience, for any positive integer $n$, we let $\phi^n=\overbrace{\phi\circ\cdots\circ\phi}^{n\text{ times}}$; furthermore, we set $\phi^0=\id_S$.
For any $a\in S$, if there is some positive integer $n$ such that $\phi^n(a)=a$, we say that $a$ is \emph{periodic} (for $\phi$).
Let $\Per{\lp S,\phi\rp}=\lb a\in S\mid a\text{ is periodic for }\phi\rb$.

The purpose of this paper is to study the number periodic points of dynamical systems induced by polynomials or rational functions of fixed degree over finite fields, so we mention the definition of the \emph{degree} of a rational function that we will use: if $k$ is a field, $\phi\in k(X)$, and $d\in\Z_{\geq0}$, we will write $\deg{\phi}=d$ if there exist $f,g\in k[X]$ such that
\begin{itemize}
\item
$\phi=f/g$,
\item
$\gcd{(f,g)}=1$, and
\item
$d=\max{\lp\lb\deg{f},\deg{g}\rb\rp}$.
\end{itemize}
For any nonnegative integer $d$ and prime power $q$, let
\begin{itemize}
\item
$\displaystyle{\mathcal{P}(d,q)\colonequals
\frac{1}{\lv\lb\phi\in \F_{q}[X]\mid\deg{\phi}=d\rb\rv}
\sum_{\substack{\phi\in\F_{q}[X]\\
\deg{\phi}=d}}
{\frac{\lv\Per{\lp\F_{q},\phi\rp}\rv}{\lv\P^1\lp\F_{q}\rp\rv}}}$,
\item$\displaystyle{\mathcal{R}(d,q)\colonequals
\frac{1}{\lv\lb\phi\in \F_{q}(X)\mid\deg{\phi}=d\rb\rv}
\sum_{\substack{\phi\in\F_{q}(X)\\
\deg{\phi}=d}}
{\frac{\lv\Per{\lp\F_{q},\phi\rp}\rv}{\lv\P^1\lp\F_{q}\rp\rv}}}$,
\end{itemize}
and let $Q_d$ be the set of prime powers coprime to $d!$.
In \cref{theorems}, we prove:
\begin{theorem}\label{thewinner}
For any $d\in\Z_{\geq2}$,

\vspace{10px}
\hspace{70px}\begin{enumerate*}[\textup{(}a\textup{)}]
\item\label{polynomials}
\,\,$\displaystyle{\lim_{\substack{q\in Q_d\\q\to\infty}}
{\mathcal{P}(d,q)}
=0}$\hspace{25px}\text{and}\hspace{25px}
\item\label{rationals}
\,\,$\displaystyle{\lim_{\substack{q\in Q_d\\q\to\infty}}
{\mathcal{R}(d,q)}
=0}$.
\end{enumerate*}
\end{theorem}

In previous work \cite{G}, we explicitly bounded $\mathcal{P}(d,q)$, but only in the case where $d=2$.
Specifically, \cite[Corollary~1.1]{G} holds when $q$ is the power of an odd prime $p$, stating that if $q>p^{6\log{p}}$, then
\[
\mathcal{P}\lp2,q\rp<\frac{22}{\log{\log{q}}}.
\]
\cref{thewinner}, on the other hand, holds for both polynomials and rational functions and allows for arbitrarily large degrees.
We pause to remark that \cite[Corollary~1.1]{G} \emph{does not} imply the $d=2$ case of \hyperref[thewinner]{Theorem~\ref*{thewinner}~\ref*{polynomials}}, since it only hold when $q$ is large with respect to $p$.
That is, the limits of \cref{thewinner} converge uniformly across all but finitely many primes.

We now describe the methods, results, and organization of this paper, after pausing to introduce some notation.
Suppose that $R$ is a commutative ring.
If $\p\in\Spec{(R)}$, we write $[R]_\p$ for $\Frac{\lp R/\p\rp}$, and if $f\in R[X]$, we write $[f]_\p$ for the image of $f$ under the $R$-algebra morphism
\[
R[X]\twoheadrightarrow\lp R/\p\rp[X]\hookrightarrow[R]_\p[X];
\]
similarly, if $a\in R$, we write $[a]_\p$ for the image of $a+\p$ under the canonical injection $R/\p\hookrightarrow[R]_\p$.
Now let $k=\Frac{(R)}$ and suppose $\phi\in k(X)$, say with polynomials $g,h\in R[X]$ satisfying $\phi=g/h$ and $[h]_\p\neq0$.
In this situation, we write $[\phi]_\p$ for $[g]_\p/[h]_\p$.
Given such an $R$ and $\phi\in k(X)$, we see that $\Spec{(R)}$ parameterizes a family of dynamical systems: indeed, to $\p\in\Spec{(R)}$ we associate the dynamical system $\lp\P^1\lp[R]_\p\rp,[\phi]_\p\rp$.
We say that $\phi$ has \emph{good reduction at $\p$} if $\deg{[\phi]_\p}=\deg{\phi}$.

Before mentioning the main result of \cref{highdimension}, we recall that an integral domain $D$ is
\begin{itemize}
\item
\emph{residually finite} if for any nonzero ideal $I$ of $D$, we have that $D/I$ is finite and
\item
a \emph{Goldman domain} if $\Frac{(D)}$ is a finitely generated $D$-algebra\footnote{If $D$ is Noetherian (which it will be in this paper), then $D$ is a Goldman domain if and only if it has Krull dimension at most one and contains only finitely many maximal ideals}.
\end{itemize}
Next, we must extend standard definitions for polynomials to include rational functions.
If $k$ is a field and $\phi\in k(X)$, we say $\phi$ is \emph{separable over $k$} if there exist $f,g\in k[X]$ such that $\phi=f/g$ and $f$ is separable over $k$.
Similarly, if $L$ is any extension of $k$, we say that $\phi$ \emph{splits in $L$} if there exist $f,g\in k[X]$ such that $\phi=f/g$ and $f$ splits in $L$.
Any field of minimal degree in which $\phi$ splits is a \emph{splitting field} of $\phi$.
Finally, we introduce some notation for group actions.
If $\alpha$ is an action of a finite group $G$ on a set $S$, we will write
\begin{itemize}
\item
$\Fix{(\alpha)}\colonequals\lb\sigma\in G\mid\text{there exists }a\in S\text{ such that }\alpha(\sigma)(a)=a\rb$ and
\item
$\fix{(\alpha)}\colonequals\lv\Fix{(\alpha)}\rv/\lv G\rv$.
\end{itemize}
To extend this notation to Galois actions, for any field $K$ and separable $\psi\in K(X)$, let $L$ be a splitting field of $\psi$ over $K$ and $\alpha$ be the action of $\Gal{\lp L/K\rp}$ on the roots of $\psi$ in $L$; we will write $\fix_K{(\psi)}$ for $\fix{(\alpha)}$, since this quantity depends only on $K$ and $\psi$.
In \cref{highdimension}, we prove \cref{onenn}.
\begin{theorem}\label{onenn}
Suppose
\begin{itemize}
\item
$D$ is a residually finite Dedekind domain that is not a Goldman domain and
\item
$R$ is a finitely-generated $D$-algebra that is an integrally closed domain.
\end{itemize}
Write $k$ for the field of fractions of $R$, let $\phi\in k(X)$, and let $n\in\Z_{\geq1}$.
Write
\begin{itemize}
\item
$K$ for $k(t)$ and
\item
$d$ for $\deg{\phi}$.
\end{itemize}
Suppose that
\begin{enumerate}[\textup{(}1\textup{)}]
\item\label{gottasep}
$\phi^\prime(X)\neq0$,
\end{enumerate}
so that $\phi^n(X)-t$ is separable over $K$ by \cref{sepduh}, and let
\begin{itemize}
\item
$L$ be the splitting field of $\phi^n(X)-t$ over $K$ and
\item
$G$ be the Galois group of $L/K$.
\end{itemize}
If
\begin{enumerate}[\textup{(}1\textup{)}]
\setcounter{enumi}{1}
\item\label{gottabeclosed}
$k$ is algebraically closed in $L$, and
\item\label{gottatame}
either $\Char{(k)}=0$ or $\gcd{\lp\Char{(k)},\lv G\rv\rp}=1$,
\end{enumerate}
then there exists a nonempty open subset $U_{\phi,n}$ of $\Spec{(R)}$ with the property that if $\m$ is a maximal ideal in $U_{\phi,n}$, then $\phi$ has good reduction at $\m$ and
\[
\lv\frac{\lv[\phi]_\m^n\lp\P^1\lp[R]_\m\rp\rp\rv}
{\lv\P^1\lp[R]_\m\rp\rv}-\fix_K{\lp\phi^n(X)-t\rp}\rv
<\frac{7nd\lv G\rv}{\lv[R]_\m\rv^{3/2}}\cdot\fix_K{\lp\phi^n(X)-t\rp}.
\]
\end{theorem}
\noindent It is our work in \cref{algebra} that ensures the hypotheses on $R$ in \cref{onenn} guarantee that for any maximal ideal $\m$ of $R$, the residue field $[R]_\m$ is a finite field; in other words, the displayed fractions in \cref{onenn} are well-defined thanks to our work in \cref{algebra}.
\cref{onenn} generalizes \cite[Theorem~3.4]{G}, which holds only when $R=D$.
The proof of \cref{onenn} requires taking mild care with Zariski-open subsets of $\Spec{(R)}$.
It is \cref{onenn} that allows us to prove \cref{thewinner} in \cref{theorems}.
In fact, the proof of \cref{thewinner} requires two steps: one using the case of \cref{onenn} where $D$ is the integers (to address all but finitely many primes) and one using the case where $D$ is a polynomial ring over a finite field (to address one by one most of the remaining primes).
In both cases, we use the Lang-Weil bound on numbers points of proper subvarieties over a finite field~\cite{LW}.

\section{Previous work}
\label{algebra}

Before proceeding, we give a very short description of previous work on the statistics of image sizes and periodic points of dynamical systems.
In \cite[Proposition~5.3]{JKMT}, the authors proved a version of \cref{onenn}, one in which $D=R$ is the ring of integers of a number field and in which the error bound is inexplicit.
Juul made this bound explicit \cite[Theorem~2.1]{JuulP}, by computing fixed point proportions of many group actions \cite[Theorem~1.1]{Juul} and by computing height bounds on iterates of critical points of polynomial dynamical systems over number fields \cite[Theorem~1.5]{JuulP} .
In \cite[Theorem~3.4]{G}, we extended this work to include the case where $D=R$ is a residually finite Dedekind domain.
Moreover, in \cite[Theorem~5.6]{G}, we generalized the height bounds of Juul to arbitrary global fields.
This work allowed for the computation of periodic proportions of families of rational functions of fixed positive characteristic, leading to the previously-mentioned \cite[Corollary~1.1]{G}, for example.

This previous work relied on $D=R$ versions of \cref{onenn}.
This constraint led to statistical results on one-parameter families of dynamical systems (often a family parameterized by a polynomial $x^d+c\in k[X]$, where $d\in\Z_{\geq2}$, $k$ is a global field, and $c\in k$).
In this paper, we use \cref{onenn} to parameterize arbitrary many coefficients of polynomials (and rational functions) of arbitrarily large degrees.
To do so, we use the fact that \cref{onenn} allows for $D$-algebras $R$ of arbitrarily large rank (although the proof of \cref{thewinner} only requires the case where $R$ is a finitely-generated \emph{free} $D$-algebra).
The proof of \cref{thewinner} requires two applications of \cref{onenn}: once with $D=\Z$ to show \cref{thewinner} holds for all but finitely many primes and once with $D=\F_q[r]$ with $\gcd{(q,d!)}=1$ to account for most of the remaining primes one by one.

\section{Algebraic preliminaries}
\label{algebra}

The purpose of this section is to prove \cref{thingsarefinite}, which provides conditions ensuring that certain residue fields of integral domains are finite.

\begin{lemma}\label{notfingen}
Suppose $D$ is a Noetherian domain,
that $R$ is a $D$-algebra, that $\m$ is a maximal ideal of $R$, and that $[R]_\m$ is a finitely-generated $D$-algebra.
If $\m\cap D=\lb0\rb$, then $D$ is a Goldman domain.
\end{lemma}
\begin{proof}
Since $\m\cap D=\lb0\rb$, we know that $[R]_\m$ is a field extension of $\Frac{(D)}$.
We claim this extension is finite; indeed, the hypothesis that $[R]_\m$ is a finitely-generated $D$-algebra allows us to apply Zariski's Lemma~\cite{Z}.
Since $D$ is Noetherian, we conclude by applying the Artin-Tate lemma~\cite{AT} to the chain $D\subseteq\Frac{(D)}\subseteq[R]_\m$.
\end{proof}

\begin{lemma}\label{thingsarefinite}
Suppose that $D$ is a residually finite Dedekind domain that is not a Goldman domain, that $R$ is a $D$-algebra, and that $\m$ is a maximal ideal of $R$.
If $[R]_\m$ is a finitely-generated $D$-algebra, then $[R]_\m$ is a finite field.
\end{lemma}
\begin{proof}
Since $D$ is not a Goldman domain, we apply \cref{notfingen} to deduce that $\m\cap D\neq\lb0\rb$; thus, the hypothesis that $D$ is a Dedekind domain ensures that $\m\cap D$ is a maximal ideal of $D$.
The hypothesis that $D$ is residually finite ensures that $[D]_{\m\cap D}$ is a finite field and the hypothesis that $[R]_\m$ is a finitely-generated $D$-algebra implies that the field extension $[R]_\m/[D]_{\m\cap D}$ is finite, so $[R]_\m$ is a finite field as well.
\end{proof}

\begin{example}\label{exercise}
In \cref{thingsarefinite}, the hypothesis that $D$ is not a Goldman domain is necessary.
Indeed, let $S$ be the set of odd integers and take $D=S^{-1}\Z$, so that $D$ is a Goldman domain.
Let $R=D[X]$ and $\m$ the kernel of the (surjective) homomorphism from $R$ to $\Q$ that sends $X$ to $2^{-1}$, so that $\m=(2X-1)R$.
Then we see that
\begin{itemize}
\item
$\m\cap D=\lb0\rb$ and
\item
$[R]_\m\simeq\Q$ is not a finite field.
\end{itemize}
\end{example}

\section{Image sizes and Galois groups}
\label{highdimension}

We begin by mentioning \cref{plentygood}, which ensures that most specializations of a rational function preserve its degree.

\begin{fact}\label{plentygood}
Suppose that $R$ is an integral domain with field of fractions $k$ and that $\phi\in k(X)$.
If we write
\[
\mathcal{R}=\lb\p\in\Spec{(R)}\mid
\phi\text{ has good reduction at }\p\rb,
\]
then $\mathcal{R}$ contains a nonempty open subset of $\Spec{(R)}$.
\end{fact}
\begin{proof}
This fact follows from~\cite[Section~2.4]{Sil}: for any $f,g\in R[X]$ such that $\phi=f/g$ and $f,g$ have no common factors in $k[X]$, if we let $r\in R\setminus\lb0\rb$ be the resultant of $f$ and $g$, then any $\p\in\Spec{(R)}\setminus\mathcal{R}$ must contain $r$.
\end{proof}

\noindent The following fact was noted in the first paragraph of~\cite[Section~3]{JKMT}.

\begin{fact}\label{sepduh}
Suppose that $k$ is a field and $\phi\in k(X)$.
If $\phi^\prime(X)\neq0$, then for all $n\in\Z_{\geq0}$, the rational function $\phi^n(X)-t$ is separable over $k(t)$.
\end{fact}

We may now state the \ref{effCheb}, versions of which have appeared many times: see the proof of Proposition~5.3 in~\cite{JKMT}, the proof of Proposition~3.3 in~\cite{JuulFPP}, and~\cite[Theorem~2.1]{JuulP}.
The version we state is a special case of \cite[Theorem~2.1]{JuulP} and is based on an effective form of the Chebotarev Density Theorem; namely, \cite[Proposition~6.4.8]{FJ}, though certain cases date back to~\cite{Reich,Rhyp}.

\begin{named}{Effective Image Size Theorem}\label{effCheb}
Suppose that $q$ is a prime power, that $\phi\in\F_q(X)$, and that $n\in\Z_{\geq1}$.
Suppose that $\phi^\prime(X)\neq0$ and let
\begin{itemize}
\item
$d=\deg{\phi}$,
\item
$L$ be a splitting field of $\phi^n(X)-t$ over $\F_q(t)$, and
\item
$G=\Gal{\lp L/\F_q(t)\rp}$.
\end{itemize}
If
\begin{itemize}
\item
$L/\F_q(t)$ is tamely ramified, and
\item
$\F_q$ is algebraically closed in $L$,
\end{itemize}
then
\[
\lv\frac{\lv\phi^n\lp\P^1\lp\F_q\rp\rp\rv}{\fix_{\F_q(t)}{\lp\phi^n(X)-t\rp}}-\lv\P^1\lp\F_q\rp\rv\rv
<\frac{7nd\lv G\rv}{q^{1/2}}.
\]
\end{named}

Before proving \cref{onenn}, we pause to introduce some further notation used in its proof.
Suppose that $A$ is an integrally closed domain and $L$ is a Galois extension of $\Frac{(A)}$, then write $B$ for the integral closure of $A$ in $L$.
If $\q$ is a prime ideal of $A$ and $\qover$ is a prime ideal of $B$ lying over $\q$, we will write
\[
D_{L,A}\lp\qover\vert\q\rp
\]
for the decomposition group of $\qover$ over $\q$ and
\[
\delta_{\qover\vert\q}\colon D_{L,A}\lp\qover\vert\q\rp\to\Aut_{[A]_\q}{\lp[B]_\qover/[A]_\q\rp}
\]
for the associated surjective homomorphism of groups.
With this notation, we state the \hyperref[GUT]{Galois Uniformity Theorem} (a special case of \cite[Corollary~3.3]{G}), which is the final tool needed for the proof of \cref{onenn}.

\begin{named}{Galois Uniformity Theorem}\label{GUT}
Keep the hypotheses of \cref{onenn}, write $A$ for $R[t]$, and write $B$ for the integral closure of $A$ in $L$.
Then subset of $\Spec{(R)}$ consisting of those primes $\p$ such that
\begin{itemize}
\item
$\p B$ is prime,
\item
$[B]_{\p B}/[A]_{\p A}$ is Galois,
\item
$\delta_{\p B\vert\p A}$ is an isomorphism, and
\item
$[R]_\p$ is algebraically closed in $[B]_{\p B}$
\end{itemize}
contains a nonempty open subset of $\Spec(R)$.
\end{named}


\begin{proof}[Proof of \cref{onenn}]
Write $A$ for $R[t]$ and $B$ for the integral closure of $A$ in $L$.
By the \hyperref[GUT]{Galois Uniformity Theorem} and \cref{plentygood}, there exists a nonempty open subset $U_1$ of $\Spec{(R)}$ such that for all $\p\in U_1$,
\begin{itemize}
\item
$\p B$ is prime, the extension $[B]_{\p B}/[A]_{\p A}$ is Galois, and $\delta_{\p B\vert\p A}$ is an isomorphism,
\item
$[R]_\p$ is algebraically closed in $\lc B\rc_{\p B}$, and
\item
$\phi$ has good reduction at $\p$.
\end{itemize}
For any such prime $\p$, since $[B]_{\p B}/[A]_{\p A}$ is the Galois splitting field of the irreducible polynomial $[\phi]_{\p A}^n\lp X\rp-[t]_{\p A}$, we know that $[\phi]_{\p A}^n\lp X\rp-[t]_{\p A}$ is separable.
As $\p B$ is prime, we know that $D_{L,A}\lp\p B\vert\p A\rp=\Gal{(L/K)}$, so $\delta_{\p B\vert\p A}$ is an isomorphism.
This isomorphism induces a conjugacy of the action $G$ on the roots of $\phi^n(X)-t$ and the action of $\Gal{\lp[B]_{\p B}/[A]_{\p A}\rp}$ on the roots of $[\phi]_{\p A}^n\lp X\rp-[t]_{\p A}$, so the fixed-point proportion of these actions are preserved.
In particular, if $\p\in U_1$, then
\[
\fix_K{\lp\phi^n\lp X\rp-t\rp}=\fix_{[A]_{\p A}}{\lp[\phi]_{\p A}^n\lp X\rp-[t]_{\p A}\rp}.
\]

Next, let $U_2$ be the open subset of $\Spec{(R)}$ consisting of those ideals that do not contain $\lv G\rv$.
By hypothesis~\ref{gottatame}, we see that $U_2$ is nonempty; for any maximal ideal $\m\in U_2$, we know by \cref{thingsarefinite} that $[R]_\m$ is finite, and by our choice of $U_2$, we see $\gcd{\lp\Char{\lp[R]_\m\rp},\lv G\rv\rp}=1$.
Thus, we set $U_{\phi,n}=U_1\cap U_2$ to ensure that for any maximal ideal $\m\in U_{\phi,n}$, the extension $\lc B\rc_{\p B}/[A]_{\p A}$ is tamely ramified.
Moreover, for all such $\m$, we know that $d=\deg{\lp[\phi]_\p\rp}$ since $\phi$ has good reduction at $\m$; the conclusion now follows from the \ref{effCheb}.
\end{proof}

\section{Periodic points of rational functions over finite fields}
\label{theorems}

Before proving \cref{thewinner}, we introduce some objects used in the proof, then state a useful fact.
If $\F$ is any field and $d\in\Z_{\geq1}$, we will choose $r_d,r_{d-1},\ldots,r_0,s_d,s_{d-1},\ldots,s_0$ independent transcendental elements over $\F$ and write
\[
k^{\text{poly}}_{\F,d}\colonequals\F\lp r_d,\ldots,r_0\rp
\]
and
\[
k^{\text{rat}}_{\F,d}\colonequals\F\lp r_d,\ldots,r_0,s_d,\ldots,s_0\rp;
\]
these are the fields we use to parameterize the coefficients of polynomials and rational functions over $\F$.
For the generic polynomial and rational function, we will write
\[
\pi_{\F,d}
=\pi_{\F,d}\lp r_d,r_{d-1},\ldots,r_0\rp(X)
\colonequals r_dX^d+r_{d-1}X^{d-1}+\cdots+r_0
\in k^{\text{poly}}_{\F,d}[X]
\]
and
\[
\rho_{\F,d}
=\rho_{\F,d}\lp r_d,r_{d-1},\ldots,r_0,s_d,s_{d-1},\ldots,s_0\rp(X)
\colonequals\frac{r_dX^d+r_{d-1}X^{d-1}+\cdots+r_0}
{s_dX^d+s_{d-1}X^{d-1}+\cdots+s_0}
\in k^{\text{rat}}_{\F,d}(X).
\]
Finally, write $K^{\text{poly}}_{\F,d}=k^{\text{poly}}_{\F,d}(t)$ and $K^{\text{rat}}_{\F,d}=k^{\text{rat}}_{\F,d}(t)$, then for $n\in\Z_{\geq0}$, let $L^{\text{poly}}_{k,d,n}$ and $L^{\text{rat}}_{k,d,n}$ be the splitting fields of $\pi_{k,d}^n(X)-t$ and $\rho_{k,d}^n(X)-t$ over $K^{\text{poly}}_{k,d}$ and $K^{\text{rat}}_{k,d}$, respectively.

Next, we recall the definition of wreath products, which we will use in the statement and proof of \cref{specializing}.
Suppose that $G,H$ are finite groups and $T$ is a finite set.
If $\tau\colon H\to \Aut_{\textsf{Set}}{(T)}$ is an action of $H$ on $T$, then we use the coordinate permutation action of $H$ on $G^{\lv T\rv}$ to construct the group $G^{\lv T\rv}\rtimes H$; we denote this group by $G\wr_\tau H$.
Now, if $S$ is a another finite set and $\sigma$ is an action of $G$ on $S$, then $G\wr_\tau H$ acts on $S\times T$ via
\[
\lp\lp g_i\rp_{i\in T},h\rp\colon(s,t)\mapsto(\sigma\lp g_t\rp(s),\tau(h)(t));
\]
we denote this action by $\sigma\wr\tau$.

We may also iterate the wreath product: if $G$ is a finite group and $\sigma$ is an action of $G$ on a finite set $S$,
\begin{itemize}
\item
write $[G]^1$ for $G$ and $[\sigma]^1$ for $\sigma$, and
\item
for any $n\in\Z_{\geq2}$, write $[G]^n$ for $[G]^{n-1}\wr_{\sigma} G$ and $[\sigma]^n$ for $[\sigma]^{n-1}\wr\sigma$.
\end{itemize}
Thus, for all $n\in\Z_{\geq1}$, we see that $[\sigma]^n$ is an action of $[G]^n$ on $S^n$; in this situation, we will write $\Fix_n{\lp\sigma\rp}$ and $\fix_n{(\sigma)}$ for $\Fix{\lp[\sigma]^n\rp}$ and $\fix{\lp[\sigma]^n\rp}$, respectively.

We now prove \cref{specializing}, which follows almost immediately from \cite[Theorem~1.1]{Juul} and \cite[Proposition~4.5]{JuulP}.

\begin{lemma}\label{specializing}
Let $k$ be a field, let $d\in\Z_{\geq2}$, and let $n\in\Z_{\geq1}$.
Then
\[
\Gal{\lp L^{\textup{rat}}_{k,d,n}/K^{\textup{rat}}_{k,d}\rp}
\simeq\lc S_d\rc^n
\]
and
\[
\fix_{K^{\textup{rat}}_{k,d}}{\lp\rho_{k,d}^n(X)-t\rp}<\frac{2}{n+2}.
\]
If $(d,\Char{(k)})\neq(2,2)$, then
\[
\Gal{\lp L^{\textup{poly}}_{k,d,n}/K^{\textup{poly}}_{k,d}\rp}
\simeq\lc S_d\rc^n.
\]
and
\[
\fix_{K^{\textup{poly}}_{k,d}}{\lp\pi_{k,d}^n(X)-t\rp}<\frac{2}{n+2}.
\]
\end{lemma}

\begin{proof}
By \cite[Lemma~2.5]{JKMT}, we know that $\Gal{\lp L^{\textup{rat}}_{k,d,n}/K^{\textup{rat}}_{k,d}\rp}$ and $\Gal{\lp L^{\textup{poly}}_{k,d,n}/K^{\textup{poly}}_{k,d}\rp}$ are isomorphic to subgroups of $\lc S_d\rc^n$.
Next, we specialize $L^\text{rat}_{k,n,d}/K^\text{rat}_{k,d}$ to $t=0,r_d=1$.
Thanks to \cite[Theorem~1.1]{Juul}, we know the Galois group of this specialization is $\lc S_d\rc^n$.
As the Galois group of this specialization is isomorphic to a subgroup of $\Gal{\lp L^{\textup{rat}}_{k,d,n}/K^{\textup{rat}}_{k,d}\rp}$, we see that $\Gal{\lp L^{\textup{rat}}_{k,d,n}/K^{\textup{rat}}_{k,d}\rp}\simeq\lc S_d\rc^n$.
The same argument holds for $\Gal{\lp L^{\textup{rat}}_{k,d,n}/K^{\textup{rat}}_{k,d}\rp}$.

To conclude, if we let $\alpha$ be the usual action of $S_d$ on $\lb1,\ldots,d\rb$, then \cite[Proposition~4.5]{JuulP} states that $\fix_n{(\alpha)}<2/(n+2)$.
\end{proof}

Before proving \cref{thewinner}, we state a lemma we will need in its proof.

\begin{lemma}\label{babystepz}
Suppose that $d\in\Z_{\geq 2}$ and $q$ is a prime power.
If $\gcd{\lp q,d!\rp}=1$, then
\vspace{10px}

\hspace{70px}\begin{enumerate*}[\textup{(}a\textup{)}]
\item\label{babypolynomials}
\,\,$\displaystyle{\lim_{j\to\infty}
{\mathcal{P}\lp d,q^j\rp}
=0}$\hspace{25px}\text{and}\hspace{25px}
\item\label{babyrationals}
\,\,$\displaystyle{\lim_{j\to\infty}
{\mathcal{R}\lp d,q^j\rp}
=0}$.
\end{enumerate*}
\end{lemma}
\noindent We postpone the proof of \cref{babystepz} until after the proof of \cref{thewinner}, since it is similar and easier.

For every prime $p$, fix for the remainder of the paper an algebraic closure of $\F_p$, and for every $j\in\Z_{\geq1}$, let $\F_{p^j}$ be its unique subfield of size $p^j$.
\begin{proof}[Proof of \cref{thewinner}]
We will prove \hyperref[rationals]{\cref*{thewinner}~\ref*{rationals}}.
(The proof of \hyperref[polynomials]{\cref*{thewinner}~\ref*{polynomials}} is similar and easier.)
We write
\begin{align*}
k&=k^\text{rat}_{\Q,d}\\
\rho&=\rho_{\Q,d}\\
D&=\Z\\
R&=\Z\lc r_d,r_{d-1},
\ldots,r_0,s_d,s_{d-1},\ldots,s_0\rc\\
K&=K^\text{rat}_{\Q,d}=\Frac{(R[t])}.
\end{align*}
Now, for any prime $p$ and positive integer $j$, we write
\begin{align*}
V_{p,j}&=\lp\F_{p^j}\rp^{2d+2},\\
W_{p,j}&=\lb\lp a_d,a_{d-1}\ldots,a_0,0,\ldots,0\rp\mid a_d,a_{d-1}\ldots,a_0\in\F_{p^j}\rb\subseteq V_{p,j},\\
R_p&=R/pR\simeq\F_p\lc r_d,r_{d-1},\ldots,r_0,s_d,s_{d-1},\ldots,s_0\rc\\
R_{p,j}&=\F_{p^j}\lc r_d,r_{d-1},\ldots,r_0,s_d,s_{d-1},\ldots,s_0\rc,
\end{align*}
and for any $\mathbf{a}=\lp a_d,a_{d-1}\ldots,a_0,b_d,b_{d-1}\ldots,b_0\rp\in V_{p,j}$, let
\begin{align*}
\F_p\lp\mathbf{a}\rp
&=\F_p\lp a_d,a_{d-1}\ldots,a_0,b_d,b_{d-1}\ldots,b_0\rp,\\
\m_{p,j,\mathbf{a}}
&=\lp r_d-a_d\rp R_{p,j}+\cdots+\lp r_0-a_0\rp R_{p,j}
+\lp s_d-b_d\rp R_{p,j}+\cdots+\lp s_0-b_0\rp R_{p,j},\\
\m_{p,\mathbf{a}}
&=\m_{p,j,\mathbf{a}}\cap R_p.
\end{align*}
Let $\m_\mathbf{a}$ be the unique maximal ideal of $R$ that maps to $\m_{p,\mathbf{a}}\subseteq R_p=R/pR$ under the canonical projection.
Finally, we will write $\rho_p$ for the image of $\rho$ in $R_p[X]$, so that for any $j\in\Z_{\geq1}$, we have $\rho_p\in R_{p,j}[X]$.
With this notation, if $\mathbf{a}\in V_{p,j}\setminus W_{p,j}$, then
\[
\rho_p\lp\mathbf{a}\rp\in\F_{p^j}(X)
\]
and
\[
\tag{i}
\label{reducingisgood}
\lp\F_{p^j},\rho_p(\mathbf{a})\rp\text{ and }\lp[R]_{\m_\mathbf{a}},[\rho]_{\m_\mathbf{a}}\rp\text{ are isomorphic}
\hspace{30px}\text{if and only if}\hspace{30px}\F_p\lp\mathbf{a}\rp=\F_{p^j}.
\]
Moreover, we recall that there are at most $2p^{j/2}$ elements of $\F_{p^j}$ whose minimal polynomial is of degree less than $j$; use this bound
to see that
\[
\tag{ii}
\label{primitive}
\lv\lb\mathbf{a}\in V_{p,j}\mid\F_p\lp\mathbf{a}\rp\neq\F_{p^j}\rb\rv
\leq 2^{2d+2}p^{j(d+1)}.
\]

Let $\epsilon\in\R_{>0}$.
In light of \cref{babystepz}, we need only find a positive integer $J$ with the property that for any prime $p$ and any $j\in\Z_{\geq1}$: if $p>J$ or $j>J$, then $\mathcal{P}\lp d,p^j\rp<\epsilon$.
By \cref{specializing}, we know that for any positive integer $n$,
\[
\Gal{\lp L^{\text{rat}}_{k,d,n}/K\rp}\simeq\lc S_d\rc^n,
\]
so we can use \cref{specializing} to produce a specific positive integer $n$, which we fix the remainder of the proof, with the property that
\[
\fix_K{\lp\rho^n(X)-t\rp}<\frac{\epsilon}{8}.
\]
Let's write $L$ for $L^{\text{rat}}_{k,d,n}$ and $G$ for $\Gal{\lp L/K\rp}$.
It follows from \cite[Proposition~3.6]{JKMT} and \cref{specializing} that $k$ is algebraically closed in $L$.
(Indeed, let $\phi=\rho^n$ and let $M$ be the splitting field of $\phi^2(X)-t$ over $K$.
By \cref{specializing}, we see that $\Gal{(M/K)}\simeq\lc S_d\rc^{2n}$.
But \cite[Proposition~3.6]{JKMT} tells us that if $k$ were not algebraically closed in $L$, then $\Gal{(M/K)}$ would be isomorphic to a proper subgroup of $\lc S_d\rc^{2n}$.)
Moreover, since $\Char{(k)}=0$, we may apply \cref{onenn} to find a nonempty open set $U\subseteq\Spec{(R)}$ with the property that if $\m$ is a maximal ideal in $U$, then $\rho$ has good reduction at $\m$ and
\[
\frac{\lv[\rho]_\m^n\lp\P^1\lp[R]_\m\rp\rp\rv}
{\lv\P^1\lp[R]_\m\rp\rv}
<\fix_K{\lp\rho^n(X)-t\rp}
+\frac{7nd\lv G\rv}{\lv[R]_\m\rv^{3/2}}\cdot\fix_K{\lp\rho^n(X)-t\rp}.
\]
Now, let's choose any $J_0\in\Z_{\geq1}$ such that $J_0>\log{\lp7nd\lv G\rv\rp}$, so that for any prime $p$ and any positive integer $j$: if either $p>7nd\lv G\rv$ or $j> J_0$, then
\[
\frac{7nd\lv G\rv}{(p^j)^{3/2}}<1.
\]
Thus, for any $\m\in U$, say with $\Char{\lp[R]_\m\rp}=p$, if either $p>7nd\lv G\rv$ or $\lc[R]_\m:\F_p\rc>J_0$, then
\[
\tag{iii}
\label{persmall}
\frac{\lv[\rho]_\m^n\lp\P^1\lp[R]_\m\rp\rp\rv}
{\lv\P^1\lp[R]_\m\rp\rv}
<\frac{\epsilon}{4}.
\]
Since $U$ is nonempty, we may choose a nonconstant $f\in R$ such that $V(f)$ contains the complement of $U$.
Then, by~\cite[Theorem~1]{LW}, there exists $A_f\in\R_{>0}$ such that for all primes $p$ that do not divide the content of $f$, and all positive integers $j$,
\[
\tag{iv}
\label{LW}
\lv\lb\av\in V_{p,j}\mid\m_{\av}\nin U\rb\rv<A_fp^{j(2d+1)}.
\]
And of course, if $\av\in V_{p,j}$ and $\m_\av\in U$, then $\av\nin W_{p,j}$, since $\rho$ has good reduction at $\m_\av$.

Finally, choose a positive integer $J$ with
\begin{enumerate}[(i)]
\setcounter{enumi}{4}
\item\label{notsmaller}
$J>\max{\lp\lb 7nd\lv G\rv,j_0,\content{(f)}\rb\rp}$
\end{enumerate}
and with the property that for any prime $p$ and all $j\in\Z_{\geq1}$: if either $p>J$ or $j>J$, then
\begin{enumerate}[(i)]
\setcounter{enumi}{5}
\item\label{mostofthem}
$\lp1-p^{-2j}\rp^{-1}\lp1-p^{-j}\rp^{-1}<2$,
\item\label{LWj}
$A_fp^{-j}\lp1-p^{-2j}\rp^{-1}\lp1-p^{-j}\rp^{-1}<\epsilon/4$, and
\item\label{primitivej}
$p^{-jd}\lp2^{2d+2}\rp\lp1-p^{-2j}\rp^{-1}<\epsilon/4$.
\end{enumerate}
Now, by \cite[Lemma~3.2]{FG} for example, we know that for any such $p,j$,
\[
\lv\lb\phi\in \F_{p^j}(X)\mid\deg{\phi}=d\rb\rv=(p^j)^{2d-1}\lp p^{2j}-1\rp.
\]
Thus, for any prime $p$ and positive integer $j$: if either $p>J$ or $j>J$, then
\begin{align*}
&\frac{1}{\lv\lb\phi\in\F_{p^j}(X)\mid\deg{\phi}=d\rb\rv}
\sum_{\substack{\phi\in\F_{p^j}(X)\\
\deg{\phi}=d}}
{\frac{\lv\Per{\lp\F_{p^j},\phi\rp}\rv}{\lv\P^1\lp\F_{p^j}\rp\rv}}\\
&<\frac{\epsilon}{4}
+\frac{1}{\lp p^j\rp^{2d-1}\lp p^{2j}-1\rp\lp p^j-1\rp}
\hspace{-20px}
\sum_{\substack{\av\in V_j\\
\F_p(\av)=\F_{p^j}\\
\rho\text{ has good reduction at }\m_\av}}
{\hspace{-20px}\frac{\lv\Per{\lp[R]_{\m_\av},[\rho]_{\m_\av}\rp}\rv}
{\lv\P^1\lp[R]_{\m_\av}\rp\rv}}
&&\text{by \hyperref[primitive]{(\ref*{primitive})}, \ref{primitivej}, and \hyperref[reducingisgood]{(\ref*{reducingisgood})}}\\
&<\frac{\epsilon}{2}
+\frac{1}{\lp p^j\rp^{2d-1}\lp p^{2j}-1\rp\lp p^j-1\rp}
\sum_{\substack{\av\in V_j\\
\F_p(\av)=\F_{p^j}\\
\m_\av\in U}}
{\frac{\lv\Per{\lp[R]_{\m_\av},[\rho]_{\m_\av}\rp}\rv}
{\lv\P^1\lp[R]_{\m_\av}\rp\rv}}
&&\text{by \ref{notsmaller}, \hyperref[LW]{(\ref*{LW})}, and \ref{LWj}}\\
&<\epsilon
&&\text{by \ref{notsmaller}, \hyperref[persmall]{(\ref*{persmall})}, and \ref{mostofthem}}.
\end{align*}
\end{proof}

To conclude, we need only prove \cref{babystepz}.
Since its proof is similar to the proof of \cref{thewinner}, we will omit some details.

\begin{proof}[Proof of \cref{babystepz}]
We will prove \hyperref[babyrationals]{\cref*{babystepz}~\ref*{babyrationals}}.
(The proof of \hyperref[babypolynomials]{\cref*{babystepz}~\ref*{babypolynomials}} is similar and easier.)
In a similar way to the proof of \cref{thewinner}, we set
\begin{align*}
k&=k^\text{rat}_{\F_q,d}\\
\rho&=\rho_{\F_q,d}\\
D&=\F_q\lc r_d\rc\\
R&=\F_q\lc r_d,r_{d-1},
\ldots,r_0,s_d,s_{d-1},\ldots,s_0\rc\\
K&=K^\text{rat}_{\F_q,d}=\Frac{(R[t])}
\end{align*}
and for any $j\in\Z_{\geq1}$, we write
\begin{align*}
V_j&=\lp\F_{q^j}\rp^{2d+2},\\
W_j&=\lb\lp a_d,a_{d-1}\ldots,a_0,0,\ldots,0\rp\mid a_d,a_{d-1}\ldots,a_0\in\F_{q^j}\rb\subseteq V_j,\\
R_j&=\F_{q^j}\lc r_d,r_{d-1},\ldots,r_0,s_d,s_{d-1},\ldots,s_0\rc,
\end{align*}
If $\mathbf{a}=\lp a_d,a_{d-1}\ldots,a_0,b_d,b_{d-1}\ldots,b_0\rp\in V_j$, write
\begin{align*}
\F_q\lp\mathbf{a}\rp
&=\F_q\lp a_d,a_{d-1}\ldots,a_0,b_d,b_{d-1}\ldots,b_0\rp,\\
\m_{j,\mathbf{a}}
&=\lp r_d-a_d\rp R_j+\cdots+\lp r_0-a_0\rp R_j
+\lp s_d-b_d\rp R_j+\cdots+\lp s_0-b_0\rp R_j,\\
\m_\mathbf{a}
&=\m_{j,\mathbf{a}}\cap R.
\end{align*}
As before, if $\mathbf{a}\in V_j\setminus W_j$, then
\[
\rho\lp\mathbf{a}\rp\in\F_{q^j}(X);
\]
moreover,
\[
\tag{I}
\label{babyreducingisgood}
\lp\F_{q^j},\rho(\mathbf{a})\rp\text{ and }\lp[R]_{\m_\mathbf{a}},[\rho]_{\m_\mathbf{a}}\rp\text{ are isomorphic}
\hspace{25px}\text{if and only if}\hspace{25px}
\F_q\lp\mathbf{a}\rp=\F_{q^j}.
\]
and
\[
\tag{II}
\label{babyprimitive}
\lv\lb\mathbf{a}\in V_j\mid k\lp\mathbf{a}\rp\neq k_j\rb\rv
\leq 2^{2d+2}q^{j(d+1)}.
\]

Let $\epsilon\in\R_{>0}$.
As before, we use \cref{specializing} to produce an $n\in\Z_{\geq1}$ such that
\[
\fix_K{\lp\rho^n(X)-t\rp}<\frac{\epsilon}{8};
\]
write $G$ for $\Gal{\lp L^{\text{rat}}_{\F_q,d,n}/K\rp}$.
Since $\gcd{\lp q,\lv S_d\rv\rp}=1$ by hypothesis, we know that $\gcd{\lp q,\lv G\rv\rp}=1$, so we may apply \cref{onenn} to find a nonempty open set $U\subseteq\Spec{(R)}$ and a positive integer $J_0$ with the property that if $\m$ is a maximal ideal in $U$, then $\rho$ has good reduction at $\m$, and if both $\m\in U$ and $\lc[R]_\m:\F_q\rc>J_0$ then
\[
\tag{III}
\label{babypersmall}
\frac{\lv[\rho]_\m^n\lp\P^1\lp[R]_\m\rp\rp\rv}
{\lv\P^1\lp[R]_\m\rp\rv}
<\frac{\epsilon}{4}.
\]
And thanks to~\cite[Theorem~1]{LW}, there exists $A\in\R_{>0}$ such that for all $j\in\Z_{\geq1}$,
\[
\tag{IV}
\label{babyLW}
\lv\lb\av\in V_j\mid\m_\av\nin U\rb\rv<Aq^{j(2d+1)}.
\]
Of course, if $\av\in V_j$ and $\m_\av\in U$, then $\av\nin W_j$, since $\rho$ has good reduction at $\m_\av$.

Now we choose any positive integer $J$ with
\begin{enumerate}[(I)]
\setcounter{enumi}{4}
\item\label{babynotsmaller}
$J>\max{\lp\lb 7nd\lv G\rv,J_0\rb\rp}$
\end{enumerate}
and with the property that for all positive integers $j$ greater than $J$,
\begin{enumerate}[(I)]
\setcounter{enumi}{5}
\item\label{babymostofthem}
$\lp1-q^{-2j}\rp^{-1}\lp1-q^{-j}\rp^{-1}<2$,
\item\label{babyLWj}
$Aq^{-j}\lp1-q^{-2j}\rp^{-1}\lp1-q^{-j}\rp^{-1}<\epsilon/4$, and
\item\label{babyprimitivej}
$q^{-jd}\lp2^{2d+2}\rp\lp1-q^{-2j}\rp^{-1}<\epsilon/4$.
\end{enumerate}
Finally, we see that for all positive integers $j$: if $j>J$, then
\begin{align*}
&\frac{1}{\lv\lb\phi\in\F_{q^j}(X)\mid\deg{\phi}=d\rb\rv}
\sum_{\substack{\phi\in\F_{q^j}(X)\\
\deg{\phi}=d}}
{\frac{\lv\Per{\lp\F_{q^j},\phi\rp}\rv}{\lv\P^1\lp\F_{q^j}\rp\rv}}\\
&<\frac{\epsilon}{4}
+\frac{1}{\lp q^j\rp^{2d-1}\lp q^{2j}-1\rp\lp q^j-1\rp}
\hspace{-20px}
\sum_{\substack{\av\in V_j\\
\F_q(\av)=\F_{q^j}\\
\rho\text{ has good reduction at }\m_\av}}
{\hspace{-20px}\frac{\lv\Per{\lp[R]_{\m_\av},[\rho]_{\m_\av}\rp}\rv}
{\lv\P^1\lp[R]_{\m_\av}\rp\rv}}
&&\text{by \hyperref[babyprimitive]{(\ref*{babyprimitive})}, \ref{babyprimitivej}, and \hyperref[babyreducingisgood]{(\ref*{babyreducingisgood})}}\\
&<\frac{\epsilon}{2}
+\frac{1}{\lp q^j\rp^{2d-1}\lp q^{2j}-1\rp\lp q^j-1\rp}
\sum_{\substack{\av\in V_j\\
\F_q(\av)=\F_{q^j}\\
\m_\av\in U}}
{\frac{\lv\Per{\lp[R]_{\m_\av},[\rho]_{\m_\av}\rp}\rv}
{\lv\P^1\lp[R]_{\m_\av}\rp\rv}}
&&\text{by \hyperref[LW]{(\ref*{babyLW})} and \ref{babyLWj}}\\
&<\epsilon
&&\text{by \ref{babynotsmaller}, \hyperref[babypersmall]{(\ref*{babypersmall})}, and \ref{babymostofthem}}.
\end{align*}
\end{proof}

\section*{Acknowledgements} \label{Acknowledgements}

As usual, I would like to thank Andrew Bridy, who is always willing to share his uncanny ability to remember who proved almost any theorem, and when they did so.
I would also like to thank the anonymous reviewers, who provided many useful comments, leading to a stronger version of \cref{thewinner}.

\bibliography{arbitrarydegree}
\bibliographystyle{amsalpha}

\end{document}